\documentclass{amsart}
\usepackage{amsmath,amsthm,amssymb,amsfonts}
\usepackage{graphicx,color,epsfig,subfigure}
\usepackage{url,verbatim,multirow}

\newtheorem{thm}{Theorem}
\newtheorem{prop}[thm]{Proposition}

\newtheorem{lem}[thm]{Lemma}
\newtheorem{cor}[thm]{Corollary}

\newtheorem{rem}[thm]{Remark}
\newtheorem{conj}[thm]{Conjecture}

\newtheorem{fact}[thm]{Fact}

 \font\xviiroman=cmr17
\def\udot{\mathbin{\ooalign{$\cup$\crcr
   \hfil\raise 8pt\hbox{\xviiroman.}\hfil\crcr}}}
\def\bigudotx#1#2{\mathop{\smash{\ooalign{$#1\bigcup$\crcr
   \hfil\raise 8pt\hbox{#2}\hfil\crcr}}\vphantom{\bigcup}}}

\newcommand{\F}{\mathcal{F}}
\newcommand{\K}{\mathcal{K}}

\newcommand{\dist}{{\rm dist}}
\newcommand{\forb}{{\rm Forb}}

\newcommand{\ovchi}{\overline{\chi}}
\newcommand{\chib}{\chi_B}

\newcommand{\vk}{V(K)}
\newcommand{\vw}{{\rm VW}}
\newcommand{\vb}{{\rm VB}}
\newcommand{\ew}{{\rm EW}}
\newcommand{\eb}{{\rm EB}}

\newcommand{\vws}{\left|{\rm VW}\right|}
\newcommand{\vbs}{\left|{\rm VB}\right|}
\newcommand{\ews}{\left|{\rm EW}\right|}
\newcommand{\ebs}{\left|{\rm EB}\right|}

\newcommand{\vwk}{{\rm VW}(K)}
\newcommand{\vbk}{{\rm VB}(K)}
\newcommand{\ewk}{{\rm EW}(K)}
\newcommand{\ebk}{{\rm EB}(K)}
\newcommand{\egk}{{\rm EG}(K)}
\newcommand{\vwks}{\left|{\rm VW}(K)\right|}
\newcommand{\vbks}{\left|{\rm VB}(K)\right|}

\newcommand{\dw}{{\rm d}_{\rm W}}
\newcommand{\db}{{\rm d}_{\rm B}}
\newcommand{\dg}{{\rm d}_{\rm G}}

\newcommand{\mk}{{\bf M}_K}

\newenvironment{proofcite}[1]{\noindent{\bf Proof of #1.\,}}{\hfill$\Box$}

\newcommand{\one}{{\bf 1}}
\newcommand{\zero}{{\bf 0}}

\newcommand{\x}{{\bf x}}

\newcommand{\arrows}{\mapsto}

\newcommand{\hh}{\mathcal{H}}

\newcommand{\E}{\mathbb{E}}
\def\textdef{\textit}

\def\ed{{\textit{ed}}}

\title{The edit distance function and symmetrization}
\author{Ryan Martin}
\address{Department of Mathematics, Iowa State University, Ames, Iowa 50011}
\email{rymartin@iastate.edu}
\thanks{This author's research partially supported by NSF grant DMS-0901008 and by an Iowa State University Faculty Professional Development grant.}

\begin{document}
\begin{abstract}
The edit distance between two graphs on the same labeled vertex set is the size of the symmetric difference of the edge sets.  The distance between a graph, $G$, and a hereditary property, $\mathcal{H}$, is the minimum of the distance between $G$ and each $G'\in\mathcal{H}$.  The edit distance function of $\mathcal{H}$ is a function of $p\in[0,1]$ and is the limit of the maximum normalized distance between a graph of density $p$ and $\mathcal{H}$.

This paper utilizes a method due to Sidorenko [\textit{Combinatorica} \textbf{13}(1), pp. 109-120], called ``symmetrization'', for computing the edit distance function of various hereditary properties.  For any graph $H$, ${\rm Forb}(H)$ denotes the property of not having an induced copy of $H$.  This paper gives some results regarding estimation of the function for an arbitrary hereditary property. This paper also gives the edit distance function for ${\rm Forb}(H)$, where $H$ is a cycle on 9 or fewer vertices.
\end{abstract}
\maketitle

\section{Introduction}

The study of the edit distance in graphs originated independently by Axenovich, K\'ezdy and the author~\cite{AKM}, Alon and Stav~\cite{AS1} and, in a different formulation, by Richer~\cite{R}.  Since then, there has been a great deal of study on the edit distance itself and on the so-called edit distance function.

\subsection{The edit distance function}
The \textdef{edit distance} between graphs $G$ and $G'$ on the same labeled vertex set is $|E(G)\triangle E(G')|$ and is denoted $\dist(G,G')$.  The distance between a graph $G$ and a property $\hh$ is
$$ \dist(G,\hh):=\min\left\{\dist(G,G') : V(G)=V(G'), G'\in\hh\right\} . $$
The \textdef{edit distance function} of a property $\hh$, denoted $\ed_{\hh}(p)$, measures the maximum distance of a density $p$ graph from $\hh$. Formally,
\begin{equation}
   \ed_{\hh}(p) = \lim_{n\rightarrow\infty}\max\left\{\dist(G,\hh) : |V(G)|=n, |E(G)|=\left\lfloor {\textstyle p\binom{n}{2}}\right\rfloor\right\}/{\textstyle\binom{n}{2}} \label{eq:ghhdef}
\end{equation}
if this limit exists.

A \textdef{hereditary property} is a family of graphs that is closed under the taking of induced subgraphs.  It is natural to study the edit distance of graphs from hereditary properties because if $H$ is an induced subgraph of $G$ and $H'$ is an induced subgraph of $G'$, then $\dist(H,H')\leq\dist(G,G')$.

A hereditary property $\hh$ is \textdef{trivial} if there is an $n_0$ such that $\hh$ has no $n_0$-vertex graph (hence, no $n$-vertex graph for $n\geq n_0$).  Otherwise, it is \textdef{nontrivial}.  If $\hh$ is a nontrivial hereditary property, then it has an $n$-vertex graph for all natural numbers $n$. Throughout this paper, all graph properties will be nontrivial hereditary properties.

In \cite{BM}, a result of Alon and Stav~\cite{AS1} is generalized to show that the limit in (\ref{eq:ghhdef}) does indeed exist for nontrivial hereditary properties and, furthermore, that is the limit of the expectation of the edit distance function for random graphs with the appropriate edge-probability:
$$ \ed_{\hh}(p) = \lim_{n\rightarrow\infty}\E[\dist(G(n,p),\hh)]/{\textstyle\binom{n}{2}} . $$ It is explicitly shown in~\cite{BM} that, for any nontrivial hereditary property $\hh$, the function $\ed_{\hh}(p)$ is continuous and concave down.  Hence, it achieves its maximum at a point we define to be $\left(p_{\hh}^*,d_{\hh}^*\right)$.  It should be noted that, for some hereditary properties, $p_{\hh}^*$ might be an interval.

For every hereditary property $\hh$, there is a family of graphs that are minimal with respect to taking induced subgraphs, which we call \textdef{forbidden graphs}. We denote $\F(\hh)$ to be the minimal (with respect to vertex-deletion) set of graphs $H$ for which
$$ \hh=\bigcap_{H\in\F(\hh)}\forb(H) . $$
If $\hh=\bigcap_{H\in\F(\hh)}\forb(H)$, then we denote $\overline{\hh}$ to be the hereditary property that is $\overline{\hh}=\bigcap_{H\in\F(\hh)}\forb(\overline{H})$.  I.e., $H\in\F(\hh)$ if and only if $\overline{H}\in\F(\overline{\hh})$.  Note that $\overline{\hh}$ does not denote the complement of $\hh$ as a set.

For background on the edit distance function, applications thereof and theoretical background, we direct the reader to Balogh and the author~\cite{BM}, Alon and Stav~\cite{AS1,AS2,AS3,AS4}, Axenovich, K\'ezdy and the author~\cite{AKM}, and Axenovich and the author~\cite{AM}.  The theoretical background upon which this is based can be traced to papers by Pr\"omel and Steger~\cite{PS1,PS2,PS3}, Bollob\'as and Thomason~\cite{BT1,BT2} and Alekseev~\cite{A}, among others.

\subsection{Main results}
The main results of this paper are Theorem~\ref{thm:complete} and Theorem~\ref{thm:cycles}, but we also develop a general theory and specific techniques which enable one to compute the edit distance function.

In Theorem~\ref{thm:complete}, we provide bounds on the edit distance function for hereditary properties that forbid a clique.  We later cite the fact that $\ed_{\hh}(p)=\ed{\overline{\hh}}(p)$ (in Theorem~\ref{thm:basic}(\ref{it:comp})) and can be applied to hereditary properties that forbid an independent set.
\begin{thm}\label{thm:complete}
   Let $\hh$ be a nontrivial hereditary property such that $\F(\hh)$ contains a complete graph and let $h$ be the minimum positive integer such that $\hh\subseteq\forb(K_h)$.  Let $\chi$ be the chromatic number of $\hh$ and $m$ be the smallest positive integer such that $\F(\hh)$ contains a complete multipartite graph with $m$ parts.  Clearly, $\chi\leq m\leq h$.
   $$ \min\left\{\frac{p}{\chi-1},\frac{1-p}{\chi-1}+\frac{2p-1}{m-1}\right\}\leq
   \ed_{\hh}(p)\leq\min\left\{\frac{p}{\chi-1},1-p+\frac{2p-1}{m-1}\right\} . $$

   In particular,
   $$ \ed_{\forb(K_h)}(p)=\frac{p}{\chi-1} . $$
\end{thm}~\\

In Theorem~\ref{thm:cycles}, equation (\ref{eq:c3}) is a trivial result, equation (\ref{eq:c4}) was proven by Marchant and Thomason~\cite{MT}.  Some related results for $C_4$ were obtained by Alon and Stav~\cite{AS2}.  Thomason~\cite{T} reports that Marchant~\cite{M} has proven equation (\ref{eq:c5}) and (\ref{eq:c7}).  We note that the problem considered in~\cite{MT} and in~\cite{M} is not edit distance but can be shown to be equivalent.
\begin{thm}\label{thm:cycles}
Let $C_h$ denote the cycle on $h$ vertices.
\begin{eqnarray}
   \ed_{\forb(C_3)}(p) & = & \frac{p}{2} \label{eq:c3} \\
   \ed_{\forb(C_4)}(p) & = & p(1-p) \label{eq:c4} \\
   \ed_{\forb(C_5)}(p) & = & \min\left\{\frac{p}{2},\frac{1-p}{2}\right\} \label{eq:c5} \\
   \ed_{\forb(C_6)}(p) & = & \min\left\{p(1-p),\frac{1-p}{2}\right\} \label{eq:c6} \\
   \ed_{\forb(C_7)}(p) & = & \min\left\{\frac{p}{2},\frac{p(1-p)}{1+p},\frac{1-p}{3}\right\} \label{eq:c7} \\
   \ed_{\forb(C_8)}(p) & = & \min\left\{\frac{p(1-p)}{1+p},\frac{1-p}{3}\right\}  \label{eq:c8} \\
   \ed_{\forb(C_9)}(p) & = & \min\left\{\frac{p}{2},\frac{1-p}{4}\right\} \label{eq:c9} \\
   \ed_{\forb(C_{10})}(p) & = & \min\left\{\frac{p(1-p)}{1+2p},\frac{1-p}{4}\right\},\quad\mbox{for $p\in [1/7,1]$.} \label{eq:c10}
\end{eqnarray}
\end{thm}

\begin{cor}\label{cor:cycles}
Let $C_h$ denote the cycle on $h$ vertices.  Then,
$$ \left(p_{\forb(C_h)}^*,d_{\forb(C_h)}^*\right)
   =\left\{\begin{array}{rrl}
           (1, & 1/2), & \mbox{if $h=3$;} \\
           (1/2, & 1/4), & \mbox{if $h=4$;} \\
           (1/2, & 1/4), & \mbox{if $h=5$;} \\
           (1/2, & 1/4), & \mbox{if $h=6$;} \\
           (\sqrt{2}-1, & 3-2\sqrt{2}), & \mbox{if $h=7$;} \\
           (\sqrt{2}-1, & 3-2\sqrt{2}), & \mbox{if $h=8$;} \\
           (1/3, & 1/6), & \mbox{if $h=9$.} \\
           ((\sqrt{3}-1)/2, & (2-\sqrt{3})/2), & \mbox{if $h=10$;} \\
           \end{array}\right. $$
\end{cor}~\\

The rest of the paper is organized as follows: Section~\ref{sec:defns} gives some of the general definitions for the edit distance function, such as colored regularity graphs. Section~\ref{sec:estim} gives some theorems with which the edit distance function can be estimated.
Section~\ref{sec:complete} contains the proof of Theorem~\ref{thm:complete}. Section~\ref{sec:pcores} defines and categorizes so-called $p$-core colored regularity graphs introduced by Marchant and Thomason~\cite{MT}. Section~\ref{sec:symmetrization} discusses the symmetrization method in general. Section~\ref{sec:cycles} proves Theorem~\ref{thm:cycles} regarding cycles. Section~\ref{sec:conc} gives some concluding remarks, a conjecture and acknowledgements.


\section{Development of the proofs}
\label{sec:defns}
\subsection{Notation}
All graphs are simple.  If $S$ and $T$ are sets, then $S+T$ denotes the disjoint union of $S$ and $T$.  If $G_1$ and $G_2$ are graphs, then $G_1+G_2$ denotes the disjoint union of the graphs and $G_1\vee G_2$ denotes the join.  If $v$ and $w$ are adjacent vertices in a graph, we denote the edge between them to be $vw$.

\subsection{Colored regularity graphs}

A \textdef{colored regularity graph (CRG)}, $K$, is a simple complete graph, together with a partition of the vertices into black and white $\vk=\vwk+\vbk$ and a partition of the edges into black, white and gray $E(K)=\ewk+\egk+\ebk$.  We say that a graph $H$ embeds in $K$, (writing $H\arrows K$) if there is a function $\varphi: V(H)\rightarrow\vk$ so that if $h_1h_2\in E(H)$, then either $\varphi(h_1)=\varphi(h_2)\in\vbk$ or $\varphi(h_1)\varphi(h_2)\in\ebk\cup\egk$ and if $h_1h_2\not\in E(H)$, then either $\varphi(h_1)=\varphi(h_2)\in\vwk$ or $\varphi(h_1)\varphi(h_2)\in\ewk\cup\egk$.

For a hereditary property of graphs, we denote $\K(\hh)$ to be the subset of CRGs such that no forbidden graph maps into $K$.  That is, $\K(\hh)=\{K : H\not\arrows K, \forall H\in\F(\hh)\}$.

In a CRG, $K$, vertex $v$ is \textdef{twin to} vertex $w$ if their neighborhoods are the same.  That is, they are twin if (a) $v$ and $w$ and $vw$ have the same color and (b) whenever $x\in V(K)-\{v,w\}$, the edges $vx$ and $wx$ are the same color.

We say that a CRG, $K'$ is formed by the \textdef{partition} of a vertex $v$ if $V(K')=V(K)\cup\{v'\}$ where, for every $x\in V(K)$, the edge $v'x$ has the same color in $K'$ as $vx$ has in $K$.  All other edges in $K'$ inherit the same color as in $K$.  We say that $K''$ is formed by the \textdef{fusion} of equivalent vertices $v$ and $v'$ by letting $V(K')=V(K)-(\{v,v'\})\cup\{v''\}$ where, for every $x\in V(K)$, the edge $v''x$ has the same color as both $vx$ and $v'x$.

Two CRGs, $K$ and $K'$ are said to be \textdef{equivalent} if $K'$ can be constructed from $K$ by the partition of vertices or fusion of twin vertices.  A CRG is \textdef{reduced} if it has no pair of equivalent vertices.  A CRG, $K'$ is an equipartition of CRG, $K$ if there is an integer $\ell$ such that each vertex in $K$ is partitioned into exactly $\ell$ vertices.

A CRG $K'$ is said to be \textdef{a sub-CRG of $K$} if $K'$ can be obtained by deleting vertices of $K$.

\subsection{The $f$ and $g$ functions}

For every hereditary property, $\hh$, the function $\ed_{\hh}(p)$ in (\ref{eq:ghhdef}), measures not only the maximum normalized edit distance among density-$p$ graphs but also the expectation of the normalized distance from $G(n,p)$. That is, Alon and Stav~\cite{AS1} prove that
$$ \ed_{\hh}(p)=\lim_{n\rightarrow\infty}\E\left[\dist(G(n,p),\hh)\right]/{\textstyle \binom{n}{2}} . $$

The normalized distance of $G(n,p)$ from a hereditary property is well-defined because the distance from $G(n,p)$ to $\hh$ is concentrated around its mean.

For every CRG, $K$, we associate two functions of $p\in [0,1]$.  The function $f$ is linear in $p$ and $g$ is found by the solution of a quadratic program..  Let $K$ have a total of $k$ vertices $\{v_1,\ldots,v_k\}$, and let $\mk(p)$ be a matrix such that the entries are:
$$ [\mk(p)]_{ij}=\left\{\begin{array}{ll}
                           p, & \mbox{if $v_iv_j\in\vwk\cup\ewk$;} \\
                           1-p, & \mbox{if $v_iv_j\in\vbk\cup\ebk$;} \\
                           0, & \mbox{if $v_iv_j\in\egk$.}
                        \end{array}\right. $$
Then, we can express the $f$ and $g$ functions over the domain $p\in[0,1]$ as follows, with $\vw=\vwk$, $\vb=\vbk$, $\ew=\ewk$ and $\eb=\ebk$:
\begin{eqnarray}
   f_K(p) & = & \frac{1}{k^2}\left[p\left(\vws+2\ews\right)+(1-p)\left(\vbs+2\ebs\right)\right] \label{eq:fdef} \\
   g_K(p) & = & \left\{\begin{array}{rrcl}
   \min & \multicolumn{3}{l}{\x^T\mk(p)\x} \\
   \mbox{s.t.} & \x^T\one & = & 1 \\
   & \x & \geq & \zero \end{array}\right. \label{eq:gdef}
\end{eqnarray}
If we denote $\one$ to be the vector of all ones, then $f_K(p)=\left(\frac{1}{k}\one\right)^T\mk(p)\left(\frac{1}{k}\one\right)$.  So, $f_K(p)\geq g_K(p)$.~\\

\begin{fact}
   The function $g$ is invariant under equivalence classes of CRGs. That is, if $K$ and $K'$ are equivalent CRGs, then $g_K(p)=g_{K'}(p)$ for all $p\in [0,1]$.
\end{fact}~\\

We can use both the $f$ and $g$ functions of CRGs to compute the edit distance function.
\begin{thm}[\cite{BM}]
   For any nontrivial hereditary property $\hh$,
   $$ \ed_{\hh}(p)=\inf_{K\in\K(\hh)}g_K(p)=\inf_{K\in\K(\hh)}f_K(p) . $$
\end{thm}

\begin{rem}\label{rem:min}
   Marchant and Thomason~\cite{MT} prove that, in fact, $\ed_{\hh}(p)=\min_{K\in\K(\hh)}g_K(p)$.  That is, that for every $p\in [0,1]$, there is a CRG, $K\in\K(\hh)$, such that $\ed_{\hh}(p)=g_K(p)$.
\end{rem}~\\

A sub-CRG, $K'$, of a CRG, $K$, is a \textdef{component} if, for all $v\in V(K')$ and all $w\in V(K)-V(K')$, the edge $vw$ is gray.  Theorem~\ref{thm:components} allows the computation of $g_K$ from the $g$ functions of its components.
\begin{thm}\label{thm:components}
   Let $K$ be a CRG with components $K^{(1)},\ldots,K^{(\ell)}$.  Then
   $$ \left(g_K(p)\right)^{-1}=\sum_{i=1}^{\ell}\left(g_{K^{(i)}}(p)\right)^{-1} . $$
\end{thm}

\begin{proof}
   The matrix $\mk(p)$ is a block-diagonal matrix.  Let $g_i=g_{K^{(i)}}(p)$ for $i=1,\ldots,\ell$ and $g=g_K(p)$.  We may first assign the total weights of the vertices in each component. Then, the relative weights of the vertices in each component is defined by that component's $g$ function.

   Let $\alpha_i$ denote the total weight that the optimal solution of (\ref{eq:gdef}) assigns to the vertices of $K^{(i)}$.  Then, we obtain the following optimization problem:
   $$ g = \left\{\begin{array}{rrcl}
   \min & \alpha_1^2g_1+\cdots+\alpha_{\ell}^2g_{\ell} & & \\
   \mbox{s.t.} & \alpha_1+\cdots+\alpha_{\ell} & = & 1 \\
   & \alpha_1,\ldots,\alpha_{\ell} & \geq & 0 \end{array}\right. $$
   Using the method of Lagrange multipliers, we see that the solution is $\alpha_i=\lambda/g_i$ for $i=1,\ldots,\ell$ and $\lambda^{-1}=\sum_{i=1}^{\ell}g_i^{-1}$.  Substituting these values gives the theorem statement.
\end{proof}~\\

Theorem~\ref{thm:components} can be applied directly to CRGs that have only gray edges. Since the $g$ function for a white vertex is $p$ and the $g$ function for a black vertex is $1-p$, we have Corollary~\ref{cor:components}:
\begin{cor}\label{cor:components}
   If $K$ is a CRG all of whose edges are gray, then
   $$ g_K(p)=\left(\frac{\vwks}{p}+\frac{\vbks}{1-p}\right)^{-1} . $$
\end{cor}~\\

Proposition~\ref{prop:nogray} gives the edit distance function for some special CRGs that have no gray edges.
\begin{prop}\label{prop:nogray} Let $K$ be a CRG on $k$ vertices and no gray edges as follows:
\begin{itemize}
   \item If all vertices are white and all edges are black, then $g_K(p)=\min\{p,1-p+(2p-1)/k\}$.
   \item If all vertices are black and all edges are white, then $g_K(p)=\min\{p+(1-2p)/k,1-p\}$.
\end{itemize}
\end{prop}~\\

\section{Estimation of the edit distance function}
\label{sec:estim}
Denote $K(r,s)$ to be the CRG with $r$ white vertices, $s$ black vertices and all gray edges.  Let $\hh$ be a hereditary property with $\hh=\bigcap_{H\in\F(\hh)}\forb(H)$.  The notion of $(r,s)$-colorability is discussed by Alon and Stav~\cite{AS2} where they focus on hereditary properties that are complement-invariant.

The \textdef{chromatic number} of $\hh$, denoted $\chi(\hh)$ or just $\chi$, where the context is clear, is $\min\left\{\chi(H) : H\in\F(\hh)\right\}$.  The \textdef{complementary chromatic number}\footnote{Unfortunately, the term ``cochromatic number'' is taken.  It should be noted that the cochromatic number, although its definition resembles that of $\chib$, is not the same parameter.} of $\hh$, denoted $\overline{\chi}(\hh)$ or $\ovchi$, is $\min\left\{\chi(\overline{H}) : H\in\F(\hh)\right\}$.  The \textdef{binary chromatic number} is $\max\left\{k+1 : \exists r,s, r+s=k,  H\not\arrows K(r,s), \forall H\in\F(\hh)\right\}$.

The \textdef{clique spectrum} of $\hh$ is the set
$$ \Gamma(\hh)\stackrel{\rm def}{=}\left\{(r,s) : H\not\arrows K(r,s), \forall H\in\F(\hh)\right\} . $$
The clique spectrum has a number of useful properties.  For example, it is monotone in the sense that if $(r,s)\in\Gamma(\hh)$ and $0\leq r'\leq r$ and $0\leq s'\leq s$, then $(r',s')\in\Gamma(\hh)$.  As a result, the clique spectrum of a hereditary property can be expressed as a Young tableau.  An \textdef{extreme point} of the clique spectrum $\Gamma$ is a pair $(r,s)\in\Gamma$ for which both $(r+1,s)\not\in\Gamma$ and $(r,s+1)\not\in\Gamma$. Let $\Gamma^*$ denote the extreme points of clique spectrum $\Gamma$. Figure~\ref{fig:cyc9young} shows the clique spectrum of the cycle $C_9$ expressed as a Young tableau, with the extreme points of the clique spectrum marked.

\begin{figure}[ht]\label{fig:cyc9young}
{\hfill\includegraphics[height=1.0in]{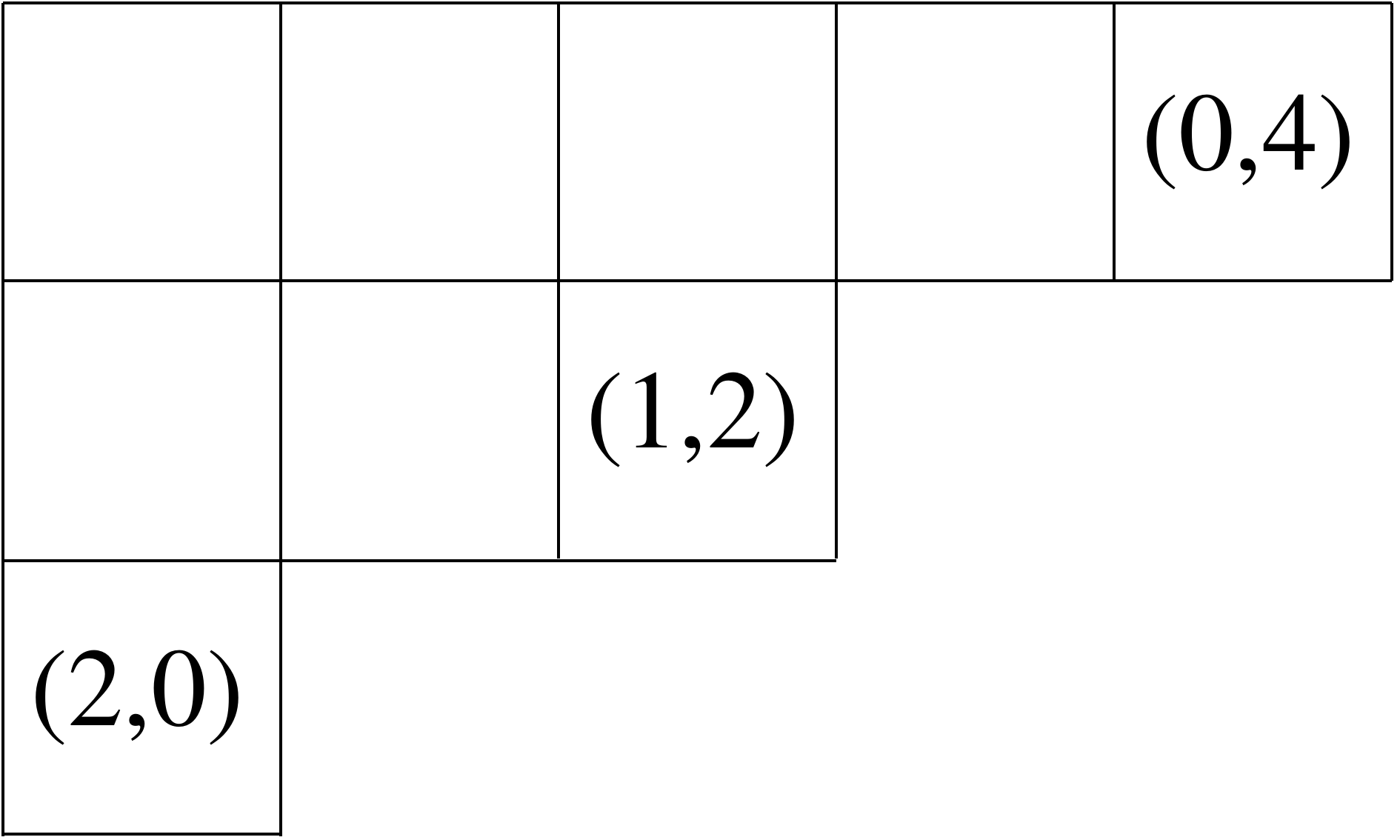}\hfill}
\caption{The clique spectrum of $C_9$ expressed as a Young tableau.  The extreme points of the clique spectrum are labeled.}
\end{figure}

\subsection{Approximating $\ed_{\hh}(p)$ by $\gamma_{\hh}(p)$}

Corollary~\ref{cor:components} gives that $g_{K(r,s)}(p)=\frac{p(1-p)}{r(1-p)+sp}$, which follows directly from Theorem~\ref{thm:components}.  Define the function $\gamma_{\hh}(p)$ as follows:
$$ \gamma_{\hh}(p)\stackrel{\rm def}{=}\min\left\{g_{K(r,s)}(p) : (r,s)\in\Gamma(\hh)\right\}=
\min\left\{\frac{p(1-p)}{r(1-p)+sp} : (r,s)\in\Gamma(\hh)\right\} . $$
Clearly, $\ed_{\hh}(p)\leq\gamma_{\hh}(p)$.  Moreover, $\gamma_{\hh}(p)=\min\left\{g_{K(r,s)}(p) : (r,s)\in\Gamma^*(\hh)\right\}$; i.e, only $(r,s)$ that are extreme points of the clique spectrum need to be used to compute $\gamma$.  The value of the function $\gamma_{\hh}(p)$ is that it is computable for any hereditary property.

\subsection{Basic observations on $\ed_{\hh}(p)$}

The following is a summary of basic facts about the edit distance function.  Item (\ref{it:bcn}) comes from Alon and Stav~\cite{AS1}.  Item (\ref{it:concon}) comes from \cite{BM}.  The remaining items are trivial.
\begin{thm}\label{thm:basic}
   Let $\hh$ be a nontrivial hereditary property with chromatic number $\chi$, complementary chromatic number $\ovchi$, binary chromatic number $\chib$ and edit distance function $\ed_{\hh}(p)$.
   \begin{enumerate}
      \item If $\chi>1$, then $\ed_{\hh}(p)\leq p/(\chi-1)$. \label{it:chi}
      \item If $\ovchi>1$, then $\ed_{\hh}(p)\leq (1-p)/(\ovchi-1)$. \label{it:ovchi}
      \item $\ed_{\hh}(1/2)=1/(2(\chib-1))=\gamma_{\hh}(1/2)$. \label{it:bcn}
      \item $\ed_{\hh}(p)$ is continuous and concave down. \label{it:concon}
      \item $\ed_{\hh}(p)=\ed_{\overline{\hh}}(1-p)$. \label{it:comp}
   \end{enumerate}
\end{thm}~\\

There are a number of immediate corollaries of Theorem~\ref{thm:basic} that help estimate the edit distance functions.  Some of the most useful are summarized in Corollary~\ref{cor:basic} and we leave the proof of them to the reader.
\begin{cor}\label{cor:basic}
   Let $\hh$ be a nontrivial hereditary property with binary chromatic number $\chib$.  Let $(r,s)$ be extreme points in the clique spectrum of $\hh$ such that $r+s=\chib$.
   \begin{enumerate}
      \item If $\chi=\chib$, then $\ed_{\hh}(p)=p/(\chi-1)$ for all $p\in [0,1/2]$.
      \item If $\chib=\ovchi$, then $\ed_{\hh}(p)=(1-p)/(\chib-1)$ for all $p\in [1/2,1]$.
      \item If $r\geq s$, then $p_{\hh}^*\geq 1/2$.
      \item If $r\leq s$, then $p_{\hh}^*\leq 1/2$.
      \item For any $(r,s)$ in the clique spectrum, $d_{\hh}^*\leq(\sqrt{r}+\sqrt{s})^{-2}$.
   \end{enumerate}
\end{cor}~\\

\section{$\hh\subseteq\forb(K_h)$}
\label{sec:complete}
In this section, we prove Theorem~\ref{thm:complete}, which bounds the edit distance function for hereditary properties that have no copy of a complete graph.  Note that $\hh\subseteq\forb(K_h)$ if and only if $K_h\in\F(\hh)$.~\\

\begin{proofcite}{Theorem~\ref{thm:complete}}
   Since $\ovchi(\hh)=1$ and $\hh$ is not trivial, $\chi(\hh)>1$.  If $K\in\K(\hh)$, then $K$ cannot have a black vertex, otherwise $K_h\arrows K$.  So, we may assume that $K\in\K(\hh)$ has all white vertices.  In every set of $\chi$ white vertices, there must be a non-gray edge. By Tur\'an's theorem, this means that $K$ has at least $\binom{k}{2}-\frac{\chi-2}{\chi-1}\cdot\frac{k^2}{2}$ non-gray edges.  Hence,
   \begin{eqnarray*}
      \ed_{\hh}(p) & \geq & f_K(p)\geq\frac{1}{k^2}\left[pk +2\min\{p,1-p\}\left(\binom{k}{2}-\frac{\chi-2}{\chi-1}\cdot\frac{k^2}{2}\right)\right] \\
      & \geq & \frac{\min\{p,1-p\}}{\chi-1}
   \end{eqnarray*}

   In every set of $m$ white vertices, there must be a white edge.  Again, by Tur\'an's theorem, $\ed_{\hh}(p)\geq f_K(p)\geq\frac{p}{m-1}$.  So, $\ed_{\hh}(p)$ is bounded below by both $p/(\chi-1)$ and the line segment connecting the points $\left(1/2,\frac{1}{2(\chi-1)}\right)$ and $\left(1,\frac{1}{m-1}\right)$.  Hence,
   $$ \ed_{\hh}(p)\geq\min\left\{\frac{p}{\chi-1},\frac{1-p}{\chi-1}+\frac{2p-1}{m-1}\right\} . $$

   As to the upper bound, we give two CRGs into which no $H\in\F(\hh)$ can map.  The first is $K^{(1)}=K(\chi-1,0)$, the CRG with $\chi-1$ white vertices and all edges gray.  By Corollary~\ref{cor:components}, $g_{K^{(1)}}(p)=p/(\chi-1)$.

   The second CRG, $K^{(2)}$, is $m-1$ white vertices and all black edges. If there were some $H\in\F(\hh)$ such that $H\arrows K^{(2)}$, then $H$ would be a complete $(m-1)$-partite graph, which is forbidden by our choice of $m$. By Proposition~\ref{prop:nogray}, $g_{K^{(2)}}(p)=\min\{p,1-p+(2p-1)/(m-1)\}$. So,
   $$ \ed_{\hh}(p)\leq\min\left\{\frac{p}{\chi-1},p,1-p+\frac{2p-1}{m-1}\right\} . $$

   The final statement comes from the observation that if $\hh=\forb(K_h)$, then $\chi=m=h$.
\end{proofcite}~\\

By Theorem~\ref{thm:basic}(\ref{it:comp}) we have the similar result for empty graphs: Let $\hh$ be a nontrivial hereditary property such that $\F(\hh)$ contains an empty graph and let $h$ be the minimum positive integer such that $\hh\subseteq\forb(\overline{K_h})$.  Let $\ovchi$ be the complementary chromatic number\footnote{The term $\ovchi(\hh)$ is, the smallest number, $k$, such that no member of $\F(\hh)$ can be partitioned into $k$ cliques. In fact, $\ovchi(\hh)=\chi(\overline{\hh})$.} of $\hh$ and $m$ be the smallest positive integer such that $\F(\hh)$ contains a $m$ disjoint cliques.  Clearly, $\ovchi\leq m\leq h$.
$$ \min\left\{\frac{p}{\ovchi-1}+\frac{1-2p}{m-1},\frac{1-p}{\ovchi-1}\right\}\leq
   \ed_{\hh}(p)\leq\min\left\{p+\frac{1-2p}{m-1},\frac{1-p}{\ovchi-1}\right\} . $$
In particular, $\ed_{\forb(\overline{K_h})}(p)=\frac{1-p}{\ovchi-1}$.~\\

\section{The $p$-core CRGs}
\label{sec:pcores}
Recall that, in Remark~\ref{rem:min} we observed that $\ed_{\hh}(p)=\min\left\{g_K(p) : K\in\K(\hh)\right\}$.  That is, for any hereditary property $\hh$ and $p\in [0,1]$, there is a CRG, $K\in\K(\hh)$ such that $\ed_{\hh}(p)=g_K(p)$.  This is found by looking at so-called $p$-core CRGs.  A CRG, $K$, is a \textdef{$p$-core CRG}, or simply a $p$-core, if $g_{K}(p)<g_{K'}(p)$ for all nontrivial sub-CRGs $K'$ of $K$.

Moreover, $p$-cores can be easily classified:
\begin{thm}[Marchant-Thomason,~\cite{MT}] \label{thm:cores}
   Let $K$ be a $p$-core CRG.
   \begin{itemize}
      \item If $p=1/2$, then $K$ has all of its edges gray.
      \item If $p<1/2$, then $\ebk=\emptyset$ and there are no white edges incident to white vertices.
      \item If $p>1/2$, then $\ewk=\emptyset$ and there are no black edges incident to black vertices.
   \end{itemize}
\end{thm}


The optimal solution to the quadratic program in (\ref{eq:gdef}) is, in some sense, regular, as described in Theorem~\ref{thm:reg}. Theorem~\ref{thm:reg} is the ``symmetrization'' referenced in the title.\footnote{Pikhurko~\cite{Pik} uses this term for the approach by Sidorenko~\cite{Sid}.} The fundamental observation is that if every optimal solution, $\x^*$, of (\ref{eq:gdef}) has no zero entries, then
\begin{equation}\label{eq:balanced}
   \mk(p)\cdot\x^*=g_K(p)\one ,
\end{equation}
where $\one$ is the all-ones vector. Of course, an optimal solutions having no zero entries corresponds, by definition, to a CRG being $p$-core and that the optimal solution to quadratic program in (\ref{eq:gdef}) is unique.

By Theorem~\ref{thm:cores}, if $K$ is a $p$-core CRG, then no edge has the same color as either of its endvertices, so we can reinterpret (\ref{eq:balanced}) as follows:
\begin{thm}[Marchant-Thomason,~\cite{MT}]\label{thm:reg}
   Let $K$ be a $p$-core CRG.  There is a unique vector $\x$ that is an optimal solution to the quadratic program in (\ref{eq:gdef}).  For all $v\in V(K)$, let the entry of $\x$ corresponding to $v$ be $\x(v)$. For each $v\in V(K)$,
   $$ g_K(p)=\x(v)\left[p\,\dw(v)+(1-p)\,\db(v)\right] , $$
   where
   $$ \dw(v)=\left\{\begin{array}{ll} \x(v), & \mbox{if $v\in\vwk$;} \\ \sum_{vz\in\ewk}\x(z), & \mbox{if $v\in\vbk$;} \end{array}\right. $$ and
   $$ \db(v)=\left\{\begin{array}{ll} \x(v), & \mbox{if $v\in\vbk$;} \\ \sum_{vz\in\ebk}\x(z), & \mbox{if $v\in\vwk$.} \end{array}\right.  $$
\end{thm}~\\

\section{Computing edit distance functions using symmetrization}
\label{sec:symmetrization}

Theorem~\ref{thm:reg}, Theorem~\ref{thm:cores}, Remark~\ref{rem:min} and the definition of $p$-cores have all of the elements in order to express $\dg(v):=1-\dw(v)-\db(v)$ for any vertex $v$ in a $p$-core CRG.  It is often useful and intuitive to focus on the gray neighborhood of vertices.
\begin{lem}\label{lem:local}
   Let $p\in (0,1)$ and $K$ be a $p$-core CRG with optimal weight function $\x$.
   \begin{enumerate}
      \item If $p\leq 1/2$, then, $\x(v)=g_K(p)/p$ for all $v\in\vw(K)$ and
      $$ \dg(v)=\frac{p-g_K(p)}{p}+\frac{1-2p}{p}\x(v) , \qquad\mbox{for all $v\in\vb(K)$.} $$ \label{it:localsmp}
      \item If $p\geq 1/2$, then $\x(v)=g_K(p)/(1-p)$ for all $v\in\vb(K)$ and
      $$ \dg(v)=\frac{1-p-g_K(p)}{1-p}+\frac{2p-1}{1-p}\x(v) , \qquad\mbox{for all $v\in\vw(K)$.} $$  \label{it:locallgp}
   \end{enumerate}
\end{lem}

\begin{proof}
We will prove the case for $p\leq 1/2$.  The case where $p\geq 1/2$ is symmetric.  Let $v\in\vw(K)$.  By Theorem~\ref{thm:cores}, all vertices are incident to $v$ via a gray edge, and by Theorem~\ref{thm:reg}, $g_K(p)=p\x(v)$.  Now let $v\in\vb(K)$.  By Theorem~\ref{thm:cores}, $v$ has no black neighbors and
$$ g_K(p)=p(1-\x(v)-\dg(v))+(1-p)\x(v) . $$
Solving for $\dg(v)$ gives the result.
\end{proof}~\\

\begin{lem}\label{lem:xbound}
   Let $p\in (0,1)$ and $K$ be a $p$-core CRG with optimal weight function $\x$.
   \begin{enumerate}
      \item If $p\leq 1/2$, then $\x(v)\leq g_K(p)/(1-p)$ for all $v\in\vb(K)$. \label{it:xbound0}
      \item If $p\geq 1/2$, then $\x(v)\leq g_K(p)/p$ for all $v\in\vw(K)$. \label{it:xbound1}
   \end{enumerate}
\end{lem}

\begin{proof}
We use the fact that $\x(v)+\dg(v)\leq 1$.  Applying Lemma~\ref{lem:local} and solving for $\x(v)$ gives the result.
\end{proof}~\\

\begin{rem}
From this point forward in the paper, if $K$ is a CRG under consideration and $p$ is fixed, $\x(v)$ will denote the weight of $v\in V(K)$ under the optimal solution of the quadratic program in equation (\ref{eq:gdef}) that defines $g_K$.
\end{rem}~\\

\section{$\forb(C_h)$, $h\in\{3,\ldots,9\}$}
\label{sec:cycles}

Thomason~\cite{T} reports that Ed Marchant has found the edit distance function for $C_5$ and $C_7$.  Here we find the function for all $C_h$, $h\in\{3,\ldots,9\}$.  The proofs in this section might be substantially similar to Marchant's.

In order to compute the edit distance function for cycles, we first make the observation that $C_3$ is a complete graph and so Theorem~\ref{thm:complete} gives Corollary~\ref{cor:c3}.
\begin{cor}\label{cor:c3}
   $$ \ed_{\forb(C_3)}(p)=p/2 . $$
   Furthermore, the only $p$-core for which this is achieved for $p\in(0,1)$ is $K(2,0)$.
\end{cor}

For $C_h$, $h\geq 4$, we first take care of easy cases so that the only $p$-cores that need to be considered have all black vertices.  We use Lemma~\ref{lem:cycles} which establishes the upper bound and eliminates all cases except when $p\leq 1/2$ and all vertices are black.
\begin{lem}\label{lem:cycles}
   Let $h\geq 4$ and $p\in (0,1)$.
   $$ \gamma_{\forb(C_h)}(p)=\left\{\begin{array}{ll}
                                       p(1-p), & \mbox{if $h=4$;} \\
                                       \min\left\{\frac{p(1-p)}{1-p+\left(\lceil h/3\rceil-1\right)p}, \frac{1-p}{\lceil h/2\rceil-1}\right\}, & \mbox{if $h\geq 6$ is even; and} \\
                                       \min\left\{\frac{p}{2}, \frac{p(1-p)}{1-p+\left(\lceil h/3\rceil-1\right)p}, \frac{1-p}{\lceil h/2\rceil-1}\right\}, & \mbox{if $h$ is odd.}
                                    \end{array}\right. $$
   Furthermore, if there is a $p$-core CRG, $K\in\K(\forb(C_h))$ such that $g_K(p)<\gamma_{\forb(C_h)}(p)$ for any $p\in(0,1)$, then $p<1/2$ and $K$ has all black vertices.
\end{lem}

\begin{proof}
   We leave it to the reader to verify that the extreme points of the clique spectrum of $\forb(C_h)$ are $\left(0,\lceil h/2\rceil-1\right)$, $\left(1,\lceil h/3\rceil-1\right)$ and, if $h$ is odd, $(2,0)$.  This establishes the value of $\gamma_{\forb(C_h)}(p)$.

   If $h=4$, the classes of possible CRGs are restricted.  If $K$ has at least 2 white vertices, they are connected via a gray or black edge and so $C_4$ would embed in $K$.  If $K$ has a white and at least two black vertices, then the edges between the white and black vertices are both gray and the edge between the black vertices is either gray or white and so $C_4$ would embed in $K$.  Thus, if $K$ has a white vertex, then it has at most one black vertex and this is $K(1,1)$, the CRG that defines $\gamma_{\forb(C_4)}(p)=p(1-p)$.  If $K$ has all white edges, then $g_K(p)=\min\{p+(1-2p)/k,1-p\}>p(1-p)$.  So, $\ed_{\forb(C_4)}(p)=p(1-p)$.

   Now, let $h\geq 5$.  Since $\gamma_{\hh}(1/2)=\ed_{\hh}(1/2)$ for all hereditary properties and $0=\gamma_{\forb(C_h)}(1)$, convexity gives that $\ed_{\hh}(p)=\frac{1-p}{\lceil h/2\rceil-1}$, for all $p\geq 1/2$.

   Finally, let $p\in(0,1/2)$ and $K$ be a $p$-core CRG such that $C_h\not\arrows K$.  If $K$ has only white vertices and $h$ is even, then $K\approx K(1,0)$ and $g_K(p)=p>\gamma_{\hh}(p)$.  If $K$ has only white vertices and $h$ is odd, then there are at most $2$ white vertices and $g_K(p)\geq p/2$ with equality if and only if $K\approx K(2,0)$.

   If $K$ has both white and black vertices, then it has at most $1$ white vertex because $C_h\arrows K(2,1)$.  Furthermore, it can have at most $\lceil h/3\rceil-1$ black vertices.  To see this, denote the vertices of $C_h$ by $\{0,1,\ldots,h-1\}$ where $0\sim 1\sim\cdots\sim h-1\sim 0$.  Let $S$ consist of the members of $\{0,\ldots,h-2\}$ that are divisible by $3$.  If $h-1$ is divisible by $3$, then add $h-2$ to $S$.  The graph $C_h-S$ has $\lceil h/3\rceil$ connected components, each of which are cliques of size $1$ or $2$.  Thus, regardless of whether the edges are white or gray, there are at most $\lceil h/3\rceil-1$ black vertices in $K$ and $g_K(p)\geq \frac{p(1-p)}{1-p+\left(\lceil h/3\rceil-1\right)p}$, with equality if and only if $K\approx K(1,\lceil h/3\rceil-1)$.

   Summarizing, if $p\in (0,1/2)$ and $g_K(p)=\ed_{\forb(C_h)}(p)$, then $K$ is either $K(0,\lceil h/2\rceil-1)$, $K(1,\lceil h/3\rceil-1)$, $K(2,0)$ and $h$ is odd, or $K$ has all black vertices (and white or gray edges).
\end{proof}~\\

From this point forward, we only restrict ourselves to $p\in (0,1/2)$ and CRGs, $K$, with only black vertices and white or gray edges because of Lemma~\ref{lem:cycles}.  We can immediately address $4$- and $5$-cycles. Corollary~\ref{cor:c4} and Corollary~\ref{cor:c5} have appeared before. Corollary~\ref{cor:c4} was proven in the proof of Lemma~\ref{lem:cycles}.
\begin{cor}[Marchant-Thomason~\cite{MT}]\label{cor:c4}
   $$ \ed_{\forb(C_4)}(p)=p(1-p) . $$
\end{cor}~\\


\begin{cor}[\cite{M}]\label{cor:c5}
   $$ \ed_{\forb(C_5)}(p)=\min\left\{\frac{p}{2},\frac{1-p}{2}\right\} . $$
\end{cor}

\begin{proof}
  Thanks to Lemma~\ref{lem:cycles}, we can restrict to $p\in (0,1/2)$ and $p$-core CRGs $K\in\K(\forb(C_5)$ for which the vertices are black.  Let $v_1$ have largest weight in $K$ and $v_2$ have largest weight in $N_G(v_1)$.  Let $g$ denote $g_K(p)$.  Since $K$ has no triangles,
  \begin{eqnarray*}
     \dg(v_1)+\dg(v_2) & \leq & 1 \\
     2\frac{p-g}{p}+\frac{1-2p}{p}(\x(v_1)+\x(v_2)) & \leq & 1 \\
     \frac{1-2p}{p}(\x(v_1)+\x(v_2)) & \leq & \frac{2g-p}{p} .
  \end{eqnarray*}
  So, $g>p/2$, a contradiction.
\end{proof}~\\

See Figure~\ref{fig:plotc4} and Figure~\ref{fig:plotc5}.

\begin{figure}[ht]\hfill%
\begin{minipage}[t]{2.25in}
\includegraphics[width=2in]{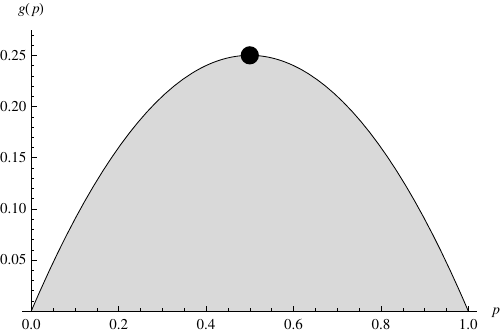}
\caption{Plot of $\ed_{\forb(C_4)}(p)=p(1-p)$.  The boundary of the shaded region is $\ed_{\forb(C_4)}(p)$.}\label{fig:plotc4}
\end{minipage}\hfill%
\begin{minipage}[t]{2.25in}
\includegraphics[width=2in]{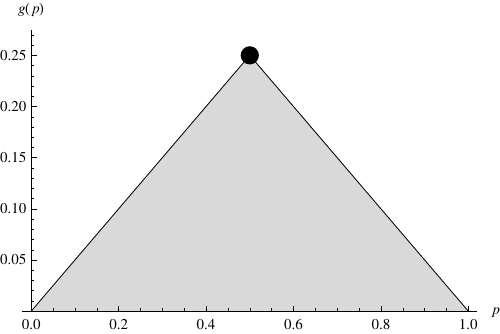}
\caption{Plot of $\ed_{\forb(C_5)}(p)=\min\{p/2,(1-p)/2\}$.}
\label{fig:plotc5}
\end{minipage}\hfill%
\end{figure}

Proposition~\ref{prop:graycycles} shows that in order to find CRGs with black vertices, white or gray edges with no $C_h$, there are many lengths of gray cycles that are forbidden in the CRG.
\begin{prop}\label{prop:graycycles}
   Let $p\in(0,1/2)$ and $K$ be a $p$-core CRG such that $K$ has black vertices and white and gray edges.  If $C_h\not\arrows K$ then $K$ has no gray cycle with length in $\left\{\lceil h/2\rceil,\ldots,h\right\}$
\end{prop}

\begin{proof}
If $C_h\arrows K$, then each vertex of $K$ receives either one or two vertices that are consecutive on the cycle.  Thus, the cycle $K$ must contain is one that corresponds to the contraction of edges of $C_h$ that map to a single black vertex of $K$.  Since these edges form a matching, the cycle required to be in $K$ has length at least $\lceil h/2\rceil$ and at most $h$.
\end{proof}

In order to deal with $\forb(C_h)$ for $h\geq 6$, we use Proposition~\ref{prop:graycycles} along with two major lemmas.  Lemma~\ref{lem:345cycle} is a general structural lemma and the results on $\forb(C_h)$ that we give are immediate corollaries.  It should be noted that if we write that a CRG, say, ``has no gray $4$-cycle,'' we mean so in the subgraph sense, so it does not contain a gray $K_4$ either.
\begin{lem}\label{lem:345cycle}
   Let $p\in (0,1/2)$ and $K$ be a $p$-core with black vertices and white or gray edges.
   \begin{enumerate}
      \item If $K$ has no gray edge, then $g_K(p)>p$. \label{it:345cycle:edge}
      \item If $K$ has neither a gray $3$-cycle nor a gray $4$-cycle, then $g_K(p)>p(1-p)$. \label{it:345cycle:no34cycle}
      \item If $K$ has no gray $3$-cycle, then $g_K(p)>p/2$. \label{it:345cycle:no3cycle}
      \item If $K$ has a gray $3$-cycle, but no gray $C_4^{+}$ (that is, four vertices that induce 5 gray edges), then $g_K(p)\geq\min\{2p/3,(1-p)/3\}$. \label{it:345cycle:3to4cycle}
      \item If $K$ has no gray $4$-cycle, then $g_K(p)>p(1-p)$ for $p\in (0,1/3)$. \label{it:345cycle:no4cycle}
      \item If $K$ has a gray $C_4^{+}$ but no gray $C_5^{++}$ (that is, five vertices that induce some $5$-cycle with two chords), then $g_K(p)>\min\{2p/3,p(1-p)/(1+p)\}$. \label{it:345cycle:4to5cycle}
      \item If $K$ has a gray chordless $4$-cycle, but no gray $K_{3,3}^{-}$ (that is, a $K_{3,3}$ missing an edge), then $g_K(p)>\min\{2p/3,2p(1-p)/(2+p)\}$. Note that $K_{3,3}^{-}$ has a $6$-cycle as a subgraph.
          \label{it:345cycle:4to6cycle}
   \end{enumerate}
\end{lem}

\begin{proof}
For ease of notation, in calculations, we sometimes let $g$ denote $g_K(p)$.
\begin{enumerate}
   \item If $K$ has no gray edges, then for any $v\in V(K)$, $g=p+(1-2p)\x(v)>p$.~\\

   \item Let $v_0\in\vk$ have the largest weight and $N_G(v_0)=\{x_1,\ldots,x_{\ell}\}$, the gray neighborhood of $v_0$.  Let $x_i=\x(v_i)$ for $i=0,1,\ldots,\ell$.  Since there are no gray triangles, there are no gray edges in $N_G(v_0)$ and since there are no gray quadrangles, $N_G(v_i)-\{v_0\}$ and $N_G(v_j)-\{v_0\}$ are disjoint for all distinct $i,j\in\{1,\ldots,\ell\}$.  So, $\{v_0\}$, $N_G(v_0)$ and each $N_G(v_i)-\{v_0\}$, $i=1,\ldots,\ell$ form a family of $\ell+2$ pairwise disjoint sets.
   \begin{eqnarray*}
      x_0+\dg(v_0)+\sum_{i=1}^{\ell}\left[\dg(v_i)-x_0\right] & \leq & 1 \\
      x_0+\dg(v_0)+\sum_{i=1}^{\ell}\left[\frac{p-g}{p}+\frac{1-2p}{p}x_i-x_0\right] & \leq & 1 \\
      x_0+\dg(v_0)+\ell\left[\frac{p-g}{p}-x_0\right]+\frac{1-2p}{p}\dg(v_0) & \leq & 1 \\
      x_0+\frac{1-p}{p}\dg(v_0)+\ell\left[\frac{p-g}{p}-x_0\right] & \leq & 1 .
   \end{eqnarray*}

   Since $x_0$ is the largest weight, $\ell\geq\dg(v_0)/x_0$ and as long as $g\geq p(1-p)$, we have $\frac{p-g}{p}-x_0\geq \frac{p-g}{p}-\frac{g}{1-p}\geq 0$ by Lemma~\ref{lem:xbound}(\ref{it:xbound0}).  Consequently,
   \begin{eqnarray}
      x_0+\frac{1-p}{p}\dg(v_0)+\frac{\dg(v_0)}{x_0}\left[\frac{p-g}{p}-x_0\right] & \leq & 1 \nonumber \\
      x_0^2+\dg(v_0)\left[\frac{p-g}{p}+\frac{1-2p}{p}x_0\right] & \leq & x_0 \nonumber \\
      x_0^2+\left[\frac{p-g}{p}+\frac{1-2p}{p}x_0\right]^2 & \leq & x_0 \nonumber \\
      \left(\frac{p-g}{p}\right)^2 +\left[2\cdot\frac{p-g}{p}\cdot\frac{1-2p}{p}-1\right] x_0
      +\left[1+\left(\frac{1-2p}{p}\right)^2\right] x_0^2 & \leq & 0 . \label{eq:345cycle:quad1}
   \end{eqnarray}

   A quadratic expression of the form $c+bx+ax^2$ with $a>0$ has a minimum value of $c-b^2/(4a)$.
   \begin{eqnarray*}
      \left(\frac{p-g}{p}\right)^2 -\frac{\left(2\cdot\frac{p-g}{p}\cdot\frac{1-2p}{p}-1\right)^2}{4\left(1+\left(\frac{1-2p}{p}\right)^2\right)} & \leq & 0 \\
      4\left(\frac{p-g}{p}\right)^2 +4\left(\frac{p-g}{p}\right)\left(\frac{1-2p}{p}\right)-1
      & \leq & 0 .
   \end{eqnarray*}
   So,
   \begin{eqnarray*}
      \frac{p-g}{p} & \leq & \frac{1}{2}\left(-\frac{1-2p}{p}+\sqrt{\left(\frac{1-2p}{p}\right)^2+1}\right) \\
      g & \geq & \frac{1}{2}\left(1-\sqrt{1-4p+5p^2}\right) .
   \end{eqnarray*}
   This expression is greater than $p(1-p)$ for all $p\in(0,1/2)$.~\\

   \item By (\ref{it:345cycle:edge}), we may assume that $K$ has a gray edge, otherwise $g_K(p)>p$.  Let $v_1v_2$ be a gray edge and $x_i=\x(v_i)$ for $i=1,2$.  Since they have no common gray neighbor,
   \begin{eqnarray*}
      \dg(v_1)+\dg(v_2) & \leq & 1 \\
      2\left(\frac{p-g}{p}\right) +\frac{1-2p}{p}(x_1+x_2) & \leq & 1
   \end{eqnarray*}
   Since $x_1+x_2>0$, we have $g>p/2$.~\\

   \item Let $\{v_1,v_2,v_3\}$ be a gray triangle in $K$ where $x_i=\x(v_i)$ for $i=1,2,3$.  Because no pairs of $v_i$ can have a common neighbor other than the remaining $v_j$,
   \begin{eqnarray*}
      \sum_{i=1}^3\left[\dg(v_i)-\left(x_1+x_2+x_3-x_i\right)\right] +\left(x_1+x_2+x_3\right) & \leq & 1 \\
      \sum_{i=1}^3\dg(v_i)-\left(x_1+x_2+x_3\right) & \leq & 1 \\
      3\left(\frac{p-g}{p}\right)+\frac{1-3p}{p}\left(x_1+x_2+x_3\right) & \leq & 1 \\
      \frac{2p}{3}+\frac{1-3p}{3}\left(x_1+x_2+x_3\right) & \leq & g .
   \end{eqnarray*}
   If $p<1/3$, then $g>2p/3$.  If $p>1/3$, then $x_1+x_2+x_3\leq 1$ implies that $g\geq (1-p)/3$.~\\

   \item Let $v_0\in\vk$ have the largest weight.  Since there are no gray quadrangles, no member of $N_G(v_0)$ has more than one gray neighbor in $N_G(v_0)$.  Let $N_G(v_0)=\{x_1,x_1',\ldots,x_m,x_m'\}\cup \{x_{2m+1},\ldots,x_{\ell}\}$, the gray neighborhood of $v_0$ such that for $i=1,\ldots,m$, $x_ix_i'$ is a gray edge.  Let $x_i=\x(v_i)$ for $i=0,1,\ldots,\ell$.  Since there are no gray quadrangles, the gray neighborhoods outside of $\{v_0\}\cup N_G(v_0)$ of distinct vertices in $N_G(v_0)$ are distinct.  Hence,
   \begin{eqnarray*}
      x_0+\dg(v_0)+\sum_{i=1}^{m}\left[\dg(v_i)+\dg(v_i')-x_i-x_i'-2x_0\right] & & \\
      +\sum_{j=2m+1}^{\ell}\left[\dg(v_j)-x_0\right] & \leq & 1 \\
      x_0+\dg(v_0)+\ell\left[\frac{p-g}{p}-x_0\right] +\sum_{i=1}^{m}\left(\frac{1-3p}{p}\right)(x_i+x_i') \\
      +\sum_{j=2m+1}^{\ell}\left(\frac{1-2p}{p}\right)x_j & \leq & 1 \\
      \ell\left[\frac{p-g}{p}-x_0\right]+x_0+\dg(v_0) +\left(\frac{1-3p}{p}\right)\dg(v_0) & \leq & 1 .
   \end{eqnarray*}

   Again, we use the fact that $\ell\geq\dg(v_0)/x_0$ and $\frac{p-g}{p}-x_0\geq 0$.
   \begin{eqnarray*}
      \frac{\dg(v_0)}{x_0}\left[\frac{p-g}{p}-x_0\right]+x_0 +\left(\frac{1-2p}{p}\right)\dg(v_0) & \leq & 1 \\
      \left(\frac{p-g}{p}\right)^2 +\left[\frac{p-g}{p}\cdot\frac{2-5p}{p}-1\right]x_0 +\left[\frac{1-2p}{p}\cdot\frac{1-3p}{p}+1\right]x_0^2 & \leq & 0.
   \end{eqnarray*}

   Optimizing over $x_0$,
   \begin{eqnarray*}
      \left(\frac{p-g}{p}\right)^2 -\frac{\left(\left(\frac{p-g}{p}\right)\left(\frac{2-5p}{p}\right)-1\right)^2}{4\left(\left(\frac{1-2p}{p}\right)\left(\frac{1-3p}{p}\right)+1\right)} & \leq & 0 \\
      \left(\frac{p-g}{p}\right)^2\left[4\frac{1-2p}{p}\cdot\frac{1-3p}{p}+4-\left(\frac{2-5p}{p}\right)^2\right] & & \\
      +2\cdot\frac{p-g}{p}\cdot\frac{2-5p}{p} -1 & \leq & 0 \\
      3\left(\frac{p-g}{p}\right)^2 +2\left(\frac{2-5p}{p}\right)\left(\frac{p-g}{p}\right) -1 & \geq & 0 .
   \end{eqnarray*}
   So,
   \begin{eqnarray*}
      \frac{p-g}{p} & \leq & \frac{1}{3}\left(-\frac{2-5p}{p}+\sqrt{\left(\frac{2-5p}{p}\right)^2+3}\right) \\
      g & \geq & \frac{2}{3}\left((1-p)-\sqrt{1-5p+7p^2}\right) .
   \end{eqnarray*}
   Some calculations show that $g>p(1-p)$ for $p\in (0,1/3)$.~\\

   \item Let the gray $C_4^{+}$ be denoted $\{v_1,v_2,v_3,v_4\}$ such that all edges are gray except, perhaps $v_1v_3$.  Let $x_i=\x(v_i)$ for $i=1,2,3,4$.  Without loss of generality, let $x_2\geq x_4$.

    No pair $(v_i,v_j)$ can have a common gray neighbor except, perhaps $(v_2,v_4)$.  Denoting $N_G(v)$ to be the set of gray neighbors of vertex $v$, the sets $N_G(v_1)-\{v_2,v_4\}$, $N_G(v_3)-\{v_2,v_4\}$ and $N_G(v_2)-\{v_1,v_3,v_4\}$ must be disjoint.  So,
    \begin{eqnarray*}
       \left(\dg(v_1)-x_2-x_4\right) +\left(\dg(v_3)-x_2-x_4\right) & & \\
       +\left(\dg(v_2)-x_1-x_3-x_4\right) +(x_1+x_2+x_3+x_4) & \leq & 1 \\
       3\cdot\frac{p-g}{p}+\frac{1-2p}{p}(x_1+x_3)+\frac{1-3p}{p}x_2-2x_4 & \leq & 1 \\
       2+\frac{1-2p}{p}(x_1+x_3)+\frac{1-3p}{p}x_2-2x_4 & \leq & \frac{3g}{p} .
    \end{eqnarray*}

    Solving for $g$,
    \begin{eqnarray*}
       g & \geq & \frac{2p}{3}+\frac{1-2p}{3}(x_1+x_3)+\frac{1-3p}{3}x_2-\frac{2p}{3}x_4 \\
       & \geq & \frac{2p}{3}+\frac{1-2p}{3}(x_1+x_3)+\frac{1-5p}{3}x_2 .
    \end{eqnarray*}

    If $p\leq 1/5$, then $g>2p/3$.  If $p>1/5$, then we use Lemma~\ref{lem:xbound}(\ref{it:xbound0}), which gives that $x_2\leq g/(1-p)$.  So,
    \begin{eqnarray*}
         g & \geq & \frac{2p}{3}+\frac{1-2p}{3}(x_1+x_4)+\frac{1-5p}{3}\left(\frac{g}{1-p}\right) \\
         & \geq & \frac{p(1-p)}{1+p}+\frac{(1-2p)(1-p)}{2(1+p)}(x_1+x_4) .
    \end{eqnarray*}
    Consequently, $g>p(1-p)/(1+p)$.~\\

   \item Let the gray $4$-cycle be denoted $\{v_1,v_2,v_3,v_4\}$ such that all edges are gray except $v_1v_3$ and $v_2v_4$.  Let $x_i=\x(v_i)$ for $i=1,2,3,4$.  If both pairs $(v_1,v_3)$ and $(v_2,v_4)$ have common neighbors outside of $\{v_1,v_2,v_3,v_4\}$, then a $K_{3,3}$ is formed.  So, suppose $v_2$ and $v_4$ have no common neighbors other than $v_1$ and $v_3$. Without loss of generality, let $x_2\geq x_4$.

    The sets $N_G(v_1)-\{v_2,v_4\}$, $N_G(v_3)-\{v_2,v_4\}$ and $N_G(v_2)-\{v_1,v_3\}$ must be disjoint.  So,
    \begin{eqnarray*}
       \left(\dg(v_1)-x_2-x_4\right) +\left(\dg(v_3)-x_2-x_4\right) & & \\
       +\left(\dg(v_2)-x_1-x_3\right) +(x_1+x_2+x_3+x_4) & \leq & 1 \\
       3\frac{p-g}{p}+\frac{1-2p}{p}(x_1+x_3)+\frac{1-3p}{p}x_2-x_4 & \leq & 1 \\
       2+\frac{1-2p}{p}(x_1+x_3)+\frac{1-3p}{p}x_2-x_4 & \leq & \frac{3g}{p} .
    \end{eqnarray*}

    Solving for $g$,
    \begin{eqnarray*}
       g & \geq & \frac{2p}{3}+\frac{1-2p}{3}(x_1+x_3)+\frac{1-3p}{3}x_2-\frac{p}{3}x_4 \\
       & \geq & \frac{2p}{3}+\frac{1-2p}{3}(x_1+x_3)+\frac{1-4p}{3}x_2 .
    \end{eqnarray*}

    If $p\leq 1/4$, then $g>2p/3$.  If $p>1/4$, then we use Lemma~\ref{lem:xbound}(\ref{it:xbound0}), which gives that $x_2\leq g/(1-p)$.
    \begin{eqnarray*}
         g & \geq & \frac{2p}{3}+\frac{1-2p}{3}(x_1+x_4)+\frac{1-4p}{3}\left(\frac{g}{1-p}\right) \\
         & \geq & \frac{2p(1-p)}{2+p}+\frac{(1-2p)(1-p)}{2+p}(x_1+x_4) .
    \end{eqnarray*}
    Consequently, $g>2p(1-p)/(2+p)$.~\\
\end{enumerate}

This concludes the proof of Lemma~\ref{lem:345cycle}.
\end{proof}~\\

\begin{cor}\label{cor:c6}
   $$ \ed_{\forb(C_6)}(p)=\min\left\{p(1-p),\frac{1-p}{2}\right\} . $$
\end{cor}

\begin{proof}
Lemma~\ref{lem:cycles} gives that the function stated above is $\gamma_{\forb(C_6)}(p)$ and so $\ed_{\forb(C_6)}(p)\leq\min\left\{p(1-p),\frac{1-p}{2}\right\}$. By Lemma~\ref{lem:cycles}, we only need to consider $p\in(0,1/2)$ and $K$ being a black-vertex $p$-core CRG in $\K(\forb(C_6))$ for which $g_K(p)<\gamma_{\forb(C_6)}(p)$. By Proposition~\ref{prop:graycycles}, $K$ has neither a $3$-cycle nor a $4$-cycle. Lemma~\ref{lem:345cycle}(\ref{it:345cycle:no34cycle}) gives that $g_K(p)\geq p(1-p)$.  So, there is no such $K$ and the corollary follows.
\end{proof}~\\

\begin{cor}\label{cor:c7}
   $$ \ed_{\forb(C_7)}(p)=\min\left\{\frac{p}{2},\frac{p(1-p)}{1+p},\frac{1-p}{3}\right\} . $$
\end{cor}

\begin{proof}
The function stated above is $\gamma_{\forb(C_7)}(p)$.  Let $p\in(0,1/2)$ and suppose $K$ is a black-vertex $p$-core CRG in $\K(\forb(C_7))$ for which $g_K(p)<\gamma_{\forb(C_7)}(p)$. By Proposition~\ref{prop:graycycles}, $K$ has no gray $4$-cycle.

Since $K$ has no gray $4$-cycle, then by Lemma~\ref{lem:345cycle}(\ref{it:345cycle:no34cycle}), either $g_K(p)>p(1-p)$ or $K$ has a gray $3$-cycle.   In terms of the former, it is trivial that this is a contradiction to $g_K(p)<\gamma_{\forb(C_7)}(p)$ for $p\in (0,1/2)$, so we assume that $G$ has a gray $3$-cycle.

If $K$ has a gray $3$-cycle but no $C_4^{+}$, then by Lemma~\ref{lem:345cycle}(\ref{it:345cycle:3to4cycle}), we have $g_K(p)>\min\{2p/3, (1-p)/3\}$. Straightforward calculations verify that this is a contradiction to $g_K(p)<\gamma_{\forb(C_7)}(p)$ for $p\in (0,1/2)$.
\end{proof}~\\

\begin{cor}\label{cor:c8}
   $$ \ed_{\forb(C_8)}(p)=\min\left\{\frac{p(1-p)}{1+p},\frac{1-p}{3}\right\} . $$
\end{cor}

\begin{proof}
The proof is the same as for Corollary~\ref{cor:c7}.
\end{proof}~\\

\begin{cor}\label{cor:c9}
   $$ \ed_{\forb(C_9)}(p)=\min\left\{\frac{p}{2},\frac{1-p}{4}\right\} . $$
\end{cor}

\begin{proof}
The function stated above is $\gamma_{\forb(C_9)}(p)$.  Let $p\in(0,1/2)$ and suppose $K$ is a black-vertex $p$-core CRG in $\K(\forb(C_9))$ for which $g_K(p)<\gamma_{\forb(C_9)}(p)$. By Proposition~\ref{prop:graycycles}, $K$ has no gray $C_5^{++}$.

Since $K$ has no gray $C_5^{++}$, then by Lemma~\ref{lem:345cycle}(\ref{it:345cycle:4to5cycle}), either $g_K(p)>\min\{2p/3,p(1-p)/(1+p)\}$ or $K$ has no gray $C_4^{+}$.  In terms of the former, straightforward calculations verify that this is a contradiction to $g_K(p)<\gamma_{\forb(C_9)}(p)$ for $p\in (0,1/2)$, so we assume that $G$ has no gray $C_4^{+}$.

If $K$ has no gray $C_4^{+}$, then by Lemma~\ref{lem:345cycle}(\ref{it:345cycle:3to4cycle}), either $g_K(p)>\min\{2p/3,(1-p)/3\}$ or $K$ has no gray $3$-cycle.  In terms of the former, it is trivial that this is a contradiction to $g_K(p)<\gamma_{\forb(C_9)}(p)$ for $p\in (0,1/2)$, so we assume that $G$ has no gray $3$-cycle. If that is the case, however, Lemma~\ref{lem:345cycle}(\ref{it:345cycle:no3cycle}) gives that $g_K(p)>p/2$, a contradiction.  So, there is no such $K$ for which $g_K(p)<\gamma_{\forb(C_9)}(p)$ and the corollary follows.
\end{proof}~\\

\begin{cor}\label{cor:c10}
   $$ \ed_{\forb(C_{10})}(p)=\min\left\{\frac{p(1-p)}{1+2p},\frac{1-p}{4}\right\}, \qquad\mbox{if $p\in[1/7,1]$} . $$
\end{cor}

\begin{proof}
The function stated above is $\gamma_{\forb(C_{10})}(p)$.  Let $p\in(0,1/2)$ and suppose $K$ is a black-vertex $p$-core CRG in $\K(\forb(C_{10}))$ for which $g_K(p)<\gamma_{\forb(C_9)}(p)$. By Proposition~\ref{prop:graycycles}, $K$ has no gray $C_5^{++}$.

Since $K$ has no gray $C_5^{++}$, then by Lemma~\ref{lem:345cycle}(\ref{it:345cycle:4to5cycle}), either $g_K(p)>\min\{2p/3, p(1-p)/(1+p)\}$ or $K$ has no gray $C_4^{+}$.  In terms of the former, straightforward calculations verify that this is a contradiction to $g_K(p)<\gamma_{\forb(C_{10})}(p)$ for $p\in [1/7,1/2)$, so we assume that $K$ has no gray $C_4^{+}$.

If $K$ has no gray $C_4^{+}$, then by Lemma~\ref{lem:345cycle}(\ref{it:345cycle:3to4cycle}), either $g_K(p)>\min\{2p/3, (1-p)/3\}$ or $K$ has no gray $3$-cycle.  In terms of the former, it is trivial that this is a contradiction to $g_K(p)<\gamma_{\forb(C_{10})}(p)$ for $p\in [1/7,1/2)$, so we assume that $K$ has no gray $3$-cycle.

If $K$ has no gray $3$-cycle, then by Lemma~\ref{lem:345cycle}(\ref{it:345cycle:no34cycle}), either $g_K(p)>p(1-p)$ or $K$ has a gray $4$-cycle. In terms of the former, it is trivial that this is a contradiction to $g_K(p)<\gamma_{\forb(C_{10})}(p)$ for $p\in (0,1/2)$, so we assume $K$ has a $4$-cycle, but since it cannot be $C_4^{+}$, it must be a gray chordless $4$-cycle.

If $K$ has a chordless gray $4$-cycle, then by Lemma~\ref{lem:345cycle}(\ref{it:345cycle:4to6cycle}), either $g_K(p)>\min\{2p/3, 2p(1-p)/(2+p)\}$ or $K$ has a gray $K_{3,3}^{-}$. In terms of the former, straightforward calculations verify that this is a contradiction to $g_K(p)<\gamma_{\forb(C_{10})}(p)$ for $p\in [1/7,1/2)$, so we assume that $K$ has a gray $K_{3,3}^{-}$.  However, as observed in Lemma~\ref{lem:345cycle}, this contains a gray $6$-cycle, which is a contradiction to $K\in\K(\forb(C_{10}))$.
\end{proof}

\begin{rem}
See Figures~\ref{fig:plotc6}-\ref{fig:plotc10} for plots of the edit distance functions described in Corollaries~\ref{cor:c6},~\ref{cor:c7},~\ref{cor:c8},~\ref{cor:c9} and~\ref{cor:c10}.
\end{rem}

\begin{figure}[ht]\hfill%
\begin{minipage}[t]{2.25in}
\includegraphics[width=2in]{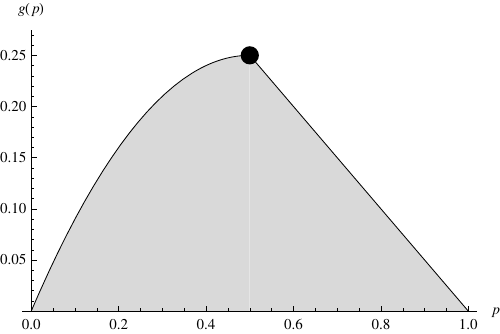}
\caption{Plot of $\ed_{\forb(C_6)}(p)=\min\{p(1-p),(1-p)/2\}$.}
\label{fig:plotc6}
\end{minipage}\hfill%
\begin{minipage}[t]{2.25in}
\includegraphics[width=2in]{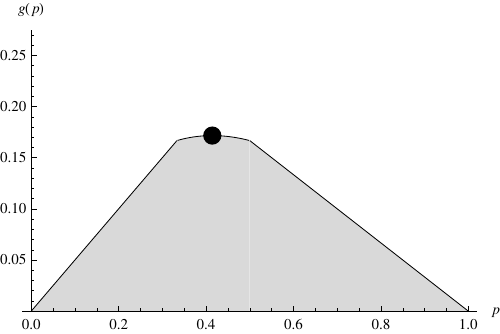}
\caption{Plot of $\ed_{\forb(C_7)}(p)=\min\{p/2,p(1-p)/(1+p),(1-p)/3\}$.}\label{fig:plotc7}
\end{minipage}\hfill%
\end{figure}

\begin{figure}[ht]\hfill%
\begin{minipage}[t]{2.25in}
\includegraphics[width=2in]{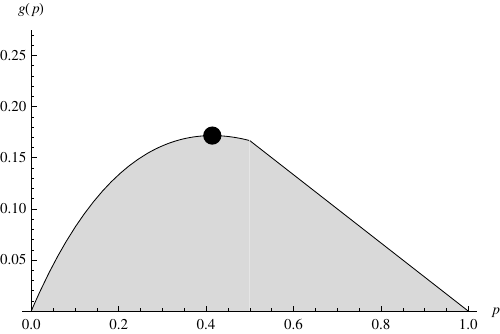}
\caption{Plot of $\ed_{\forb(C_8)}(p)=\min\{p(1-p)/(1+p),(1-p)/3\}$.}
\end{minipage}\hfill%
\begin{minipage}[t]{2.25in}
\includegraphics[width=2in]{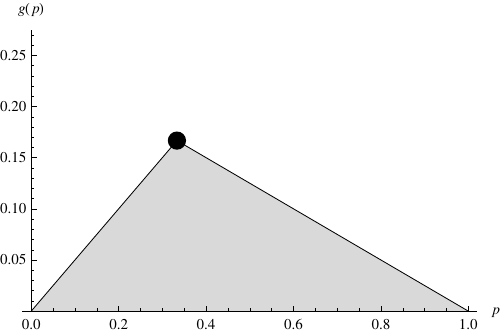}
\caption{Plot of $\ed_{\forb(C_9)}(p)=\min\{p/2,(1-p)/4\}$.} \label{fig:plotc9}
\end{minipage}\hfill%
\end{figure}

\begin{figure}[ht]
\begin{center}
\includegraphics[width=2in]{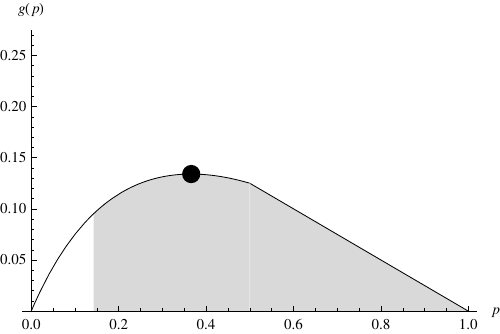}
\caption{Plot of $\ed_{\forb(C_{10})}(p)=\min\{p(1-p)/(1+2p),(1-p)/4\}$.
 An upper bound for $p<1/7$ is also on the graph.} \label{fig:plotc10}
\end{center}
\end{figure}

\section{Conclusions}
\label{sec:conc}



\subsection{$\forb\left(G(n_0,p_0)\right)$}
We provide a conjecture with some interesting implications.  Recall that $G(n,p)$ denotes the Erd\H{o}s-R\'enyi random graph on $n$ vertices with edge-probability $p$. The hereditary property $\hh=\forb(G(n_0,p_0))$ is a random variable.
\begin{conj}\label{conj:random}
   Fix $p_0\in (0,1)$ and let $\hh=\forb(G(n_0,p_0))$.  Then
   $$ \ed_{\hh}(p)=(1+o(1))\frac{2\log_2 n_0}{n_0}\min\left\{\frac{p}{-\log_2 (1-p_0)},\frac{1-p}{-\log_2 p_0}\right\} $$
   with probability approaching 1 as $n_0\rightarrow\infty$.
\end{conj}

The functions that define this bound are of the form $p/(\chi-1)$ and $(1-p)/(\ovchi-1)$.

Conjecture~\ref{conj:random} was proved for the case $p_0=1/2$ by Alon and Stav~\cite{AS2}. If it is true in general, then it implies that $p_{\hh}^*=\frac{\log (1-p_0)}{\log p_0(1-p_0)}$, which is only equal to $p_0$ itself when $p_0\in\{0,1/2,1\}$.  Recall that $\ed_{\hh}(p)= \lim_{n\rightarrow\infty}\dist(G(n,p),\hh)/\binom{n}{2}$ and it achieves its maximum at $p_{\hh}^*$.  Informally, the conjecture implies that it is harder to edit away copies of $G(n_0,p_0)$ from $G(n,p_{\hh}^*)$ than it is from $G(n,p_0)$.  This seems to be rather counterintuitive.

If Conjecture~\ref{conj:random} is false, then it implies that there is more information about the structure of random graphs than is revealed by just the chromatic numbers.

\subsection{Thanks}
I would like to thank Maria Axenovich and J\'ozsef Balogh for conversations which have improved the results.  I would like to thank Andrew Thomason for some useful conversations and for directing me to \cite{MT}.  I would also like to thank Tracy McKay for valuable discussions which deepened my understanding of previous results.

Thank you to Ed Marchant for finding an error in a previous version of this manuscript.

Figures are made by Mathematica and WinFIGQT.


\begin{thebibliography}{99}
\bibitem{A} V.E. Alekseev, On the entropy values of hereditary properties, \textit{Discrete Math. Appl.} \textbf{3} (1993), 191--199.
\bibitem{AS1} N. Alon and A. Stav, What is the furthest graph from a hereditary property? \textit{Random Structures Algorithms} \textbf{33} (2008), no. 1, pp. 87--104.
\bibitem{AS2} N. Alon and A. Stav, The maximum edit distance from hereditary graph properties. \textit{J. Combin. Th. Ser. B} \textbf{98} (2008), no. 4, pp. 672--697.
\bibitem{AS3} N. Alon and A. Stav, Stability type results for hereditary properties. \textit{J. Graph Theory} \textbf{62} (2009), no. 1, 65--83.
\bibitem{AS4} N. Alon and A. Stav, Hardness of edge-modification problems. \textit{Theoret. Comput. Sci.} \textbf{410} (2009), no. 47-49, 4920--4927.
\bibitem{AKM} M. Axenovich, A. K\'ezdy and R. Martin, On the editing distance of graphs, \textit{J. Graph Theory} \textbf{58} (2008), no. 2, 123--138.
\bibitem{AM} M. Axenovich and R. Martin, Avoiding patterns in matrices via a small number of changes. \textit{SIAM J. Discrete Math.} \textbf{20} (2006), no. 1, 49--54 (electronic).
\bibitem{BM} J. Balogh and R. Martin, Edit distance and its computation. \textit{Electron. J. Combin.} \textbf{15} (2008), no. 1, Research paper 20, 27pp.
\bibitem{BT1} B. Bollob\'as and A. Thomason, Hereditary and monotone properties of graphs. \textit{The mathematics of Paul Erd\H{o}s, II}, 70--78, Algorithms Combin., \textbf{14}, \textit{Springer, Berlin}, 1997.
\bibitem{BT2} B. Bollob\'as and A. Thomason, The structure of hereditary properties and colourings of random graphs. \textit{Combinatorica} \textbf{20} (2000), no. 2, 173--202.
\bibitem{B} W.G. Brown, On graphs that do not contain a Thomsen graph, \textit{Canad. Math. Bull.} \textbf{9} (1966), 281--285.
\bibitem{M} E. Marchant, (in preparation).
\bibitem{MT} E. Marchant and A. Thomason, Extremal graphs and multigraphs with two weighted colours, preprint.
\bibitem{Pik} O. Pikhurko, An exact Tur\'an result for the generalized triangle, \textit{Combinatorica} \textbf{28} (2008), no. 2, 187--208.
\bibitem{PS1} H.J. Pr\"omel and A. Steger, Excluding induced subgraphs: quadrilaterals, \textit{Random Structures Algorithms} \textbf{2} (1991), 55--71.
\bibitem{PS2} H.J. Pr\"omel and A. Steger, Excluding induced subgraphs II: extremal graphs, \textit{Discrete Appl. Math.} \textbf{44} (1993), 283--294.
\bibitem{PS3} H.J. Pr\"omel and A. Steger, Excluding induced subgraphs III: a general asymptotic, \textit{Random Structures Algorithms} \textbf{3} (1992), 19--31.
\bibitem{R} D.C. Richer, Ph.D. thesis, University of Cambridge (2000).
\bibitem{Sid} A.F. Sidorenko, Boundedness of optimal matrices in extremal multigraph and digraph problems, \textit{Combinatorica} \textbf{13} (1993), no. 1, 109--120.
\bibitem{T} A. Thomason, private communication.
\end{thebibliography}
\end{document}